\documentclass[11pt,reqno,tbtags,a4paper]{amsart}
\usepackage{amssymb}
\usepackage{url}
\usepackage[square,numbers]{natbib}
\bibpunct[, ]{[}{]}{;}{n}{,}{,}

\title[Hidden Words]
{Hidden Words Statistics for Large Patterns}

\date{20 March 2020}
\author{Svante Janson}
\address{Department of Mathematics, Uppsala University, PO Box 480,
SE-751~06 Uppsala, Sweden}
\email{svante.janson@math.uu.se}
\urladdr{http://www.math.uu.se/svante-janson}
\author{Wojciech Szpankowski}
\address{Center for Science of Information, Department of Computer Science, 
Purdue University, West Lafayette, IN, USA}
\email{spa@cs.purdue.edu}
\urladdr{http://www.cs.purdue.edu/homes/spa}
\thanks{The first author was supported by the Knut and Alice Wallenberg
  Foundation.
The second author was supported by
NSF Center for Science of Information (CSoI) Grant CCF-0939370,
and in addition by NSF Grant CCF-1524312.}
\keywords{Hidden pattern matching, subsequences, probability, U-statistics, projection method}
\numberwithin{equation}{section}

\renewcommand\le{\leqslant}
\renewcommand\ge{\geqslant}
\renewcommand\leq{\leqslant}

\allowdisplaybreaks

\theoremstyle{plain}% default
\newtheorem{theorem}{Theorem}[section]
\newtheorem{lemma}[theorem]{Lemma}

\newtheorem{corollary}[theorem]{Corollary}
\newtheorem{conjecture}[theorem]{Conjecture}

\theoremstyle{definition}

\newtheorem{exampleqqq}[theorem]{Example}
\newenvironment{example}{\begin{exampleqqq}}
  {\hfill\qedsymbol\end{exampleqqq}}
\newtheorem{remarkqqq}[theorem]{Remark}
\newenvironment{remark}{\begin{remarkqqq}}
  {\hfill\qedsymbol\end{remarkqqq}}
\newenvironment{romenumerate}[1][-10pt]{% optional argument changes indentation
\addtolength{\leftmargini}{#1}\begin{enumerate}% gives (i), (ii) etc.
 \renewcommand{\labelenumi}{\textup{(\roman{enumi})}}%
 }{\end{enumerate}}

\newcounter{oldenumi}
{\setcounter{oldenumi}{\value{enumi}}
\begin{romenumerate} \setcounter{enumi}{\value{oldenumi}}}
{\end{romenumerate}}

\newcounter{thmenumerate}

\newcounter{xenumerate}   %no left indentation; thus wider lines

\newcommand\pfitemx[1]{\par#1:}
\newcommand\pfitemref[1]{\pfitemx{\ref{#1}}}

\newcommand{\refT}[1]{Theorem~\ref{#1}}

\newcommand{\refL}[1]{Lemma~\ref{#1}}

\newcommand{\refR}[1]{Remark~\ref{#1}}

\newcommand{\refS}[1]{Section~\ref{#1}}

\newcommand{\refSS}[1]{Section~\ref{#1}}

\newcommand{\refE}[1]{Example~\ref{#1}}

\newcommand{\refConj}[1]{Conjecture~\ref{#1}}

\newcommand\marginal[1]{\marginpar[\raggedleft\tiny #1]{\raggedright\tiny#1}}

\newcommand\REM[1]{{\raggedright\texttt{[#1]}\par\marginal{XXX}}}
\newcommand\XREM[1]{\relax}

\begingroup
  \divide\count255 by 60
  \multiply\count255 by -60
  \advance\count255 by \time
  \ifnum \count255 < 10 \xdef\klockan{\the\count1.0\the\count255}
  \else\xdef\klockan{\the\count1.\the\count255}\fi
\endgroup

\newcommand{\sumin}{\sum_{i=1}^n}

\newcommand{\sumkn}{\sum_{k=1}^n}

\newcommand{\sumjm}{\sum_{j=1}^m}
\newcommand{\sumkm}{\sum_{k=1}^m}

\newcommand{\prodjm}{\prod_{j=1}^m}

\newcommand\set[1]{\ensuremath{\{#1\}}}

\newcommand\xpar[1]{(#1)}
\newcommand\bigpar[1]{\bigl(#1\bigr)}
\newcommand\Bigpar[1]{\Bigl(#1\Bigr)}
\newcommand\biggpar[1]{\biggl(#1\biggr)}

\newcommand\bigsqpar[1]{\bigl[#1\bigr]}

\newcommand\xcpar[1]{\{#1\}}

\newcommand\abs[1]{|#1|}
\newcommand\bigabs[1]{\bigl|#1\bigr|}
\newcommand\Bigabs[1]{\Bigl|#1\Bigr|}

\def\rompar(#1){\textup(#1\textup)}    % usage: \rompar(...)

\newcommand\Bigparfrac[2]{\Bigpar{\frac{#1}{#2}}}

\def\xexp(#1){e^{#1}}
\newcommand\ceil[1]{\lceil#1\rceil}
\newcommand\lrceil[1]{\left\lceil#1\right\rceil}

\newcommand\nn{[n]}
\newcommand\mm{[m]}
\newcommand\ntoo{\ensuremath{{n\to\infty}}}

\newcommand\mtoo{\ensuremath{{m\to\infty}}}

\newcommand\norm[1]{\|#1\|}
\newcommand\bignorm[1]{\bigl\|#1\bigr\|}
\newcommand\Bignorm[1]{\Bigl\|#1\Bigr\|}

\newcommand\Bignormm[1]{\Bignorm{#1}_2}

\newcommand\punkt{.\spacefactor=1000}    % om problem!
\newcommand\iid{i.i.d\punkt}    
\newcommand\ie{i.e\punkt}

\newcommand\cf{cf\punkt}

\newcommand{\tend}{\longrightarrow}
\newcommand\dto{\overset{\mathrm{d}}{\tend}}
\newcommand\pto{\overset{\mathrm{p}}{\tend}}

\newcommand\eqd{\overset{\mathrm{d}}{=}}

\newcommand\op{o_{\mathrm p}}
\newcommand\Op{O_{\mathrm p}}

\newcommand\bbR{\mathbb R}

\newcounter{CC}
\newcounter{cc}
\renewcommand\Re{\operatorname{Re}}

\newcommand\E{\operatorname{\mathbb E{}}}
\renewcommand\P{\operatorname{\mathbb P{}}}

\newcommand\Var{\operatorname{Var}}
\newcommand\Cov{\operatorname{Cov}}

\newcommand\Bin{\operatorname{Bin}}
\newcommand\Be{\operatorname{Be}}

\newcommand\ga{\alpha}
\newcommand\gb{\beta}
\newcommand\gd{\delta}

\newcommand\gf{\varphi}
\newcommand\gam{\gamma}
\newcommand\gamm{\gamma^2}
\newcommand\gG{\Gamma}

\newcommand\go{\omega}
\newcommand\gO{\Omega}
\newcommand\gs{\sigma}

\newcommand\gss{\sigma^2}

\renewcommand\phi{\xxx}  %% WARNING

\newcommand\cA{\mathcal A}

\newcommand\cF{\mathcal F}

\newcommand\cS{{\mathcal S}}

\newcommand\cV{\mathcal V}

\newcommand\indic[1]{\boldsymbol1\xcpar{#1}}

\newcommand\qw{^{-1}}

\newcommand\qq{^{1/2}}
\newcommand\qqw{^{-1/2}}
\newcommand\qqq{^{1/3}}

\newcommand\oi{\ensuremath{[0,1]}}

\newcommand\setoi{\set{0,1}}

\newcommand\dd{\,\mathrm{d}}

\newcommand\Z{Z}
\newcommand\Zx{\Z^*}
\newcommand\Na{N_a}
\newcommand\pa{p_a}
\newcommand\px{p_*}
\newcommand\gxx{\widehat\gs^2}
\newcommand\tauu{\tau^2}
\newcommand\gsw{\gs_1}
\newcommand\gssw{\gsw^2}
\newcommand\cbg{c(\gb,\gam)}

\newcommand\VV[1]{V_{1,#1}}
\newcommand\VVi{\VV{i}}
\newcommand\VVxi{\VV{i}'}
\newcommand\vektor{\mathbf}
\newcommand\vp{\vektor{p}}
\newcommand\vq{\vektor{q}}
\newcommand\vr{\vektor{r}}
\newcommand\vx{\vektor{x}}
\newcommand\sumaa{\sum_{a\in\cA}}
\newcommand\xin{\xi_1\dotsm\xi_n}
\newcommand\ASN{\operatorname{AsN}}
\newcommand\norme{\norm}
\newcommand\HG{\operatorname{HGe}}
\newcommand{\Polya}{P\'olya}
\newcommand\CS{Cauchy--Schwarz}
\newcommand\CSineq{\CS{} inequality}

\begin{document}

\begin{abstract} 
We study here the so called {\it subsequence pattern matching} also known
as {\it hidden pattern matching} in which one searches
for a given pattern $w$ of length $m$  as a
{\em subsequence\/} in a random text of length $n$.  
The quantity of interest is the number of occurrences of $w$ as
a subsequence (i.e., occurring in {\it not} necessarily
consecutive text locations).
This problem finds many applications from intrusion detection,
to trace reconstruction, to deletion channel, and to DNA-based storage systems.
In all of these applications, the pattern $w$ is of variable length.
To the best of our knowledge this problem was only tackled 
for a fixed length $m=O(1)$ \cite{fsv06}. In our main result
Theorem~\ref{T2a} we prove that 
for $m=o(n^{1/3})$ the number of subsequence occurrences is normally
distributed. In addition, in Theorem~\ref{Tka} we show that under some constraints on 
the structure of $w$ the asymptotic normality can be extended to
$m=o(\sqrt{n})$. For a special pattern $w$ consisting of the same symbol,
we indicate that for $m=o(n)$ the distribution of number of subsequences is
either asymptotically normal or
asymptotically log normal.  
We conjecture that this dichotomy is true for all patterns.
We use Hoeffding's projection method for $U$-statistics to prove our findings.
\end{abstract}

\maketitle

\section{Introduction and Motivation}\label{S:intro}

One of the most interesting and least studied problem in pattern matching is
known as the {\it subsequence string matching} or the {\it hidden pattern matching}
\cite{js-book}. In this case, we search for a pattern $w=w_1w_2\cdots w_m$ of
length $m$ in the text $\Xi^n=\xi_1\dots \xi_n$ of length $n$ as {\it subsequence},
that is, we are looking for indices $1\leq i_1 < i_2 < \cdots < i_m \leq n$ such that
$\xi_{i_1}=w_1, \xi_{i_2}=w_2, \ldots, \xi_{i_m}=w_m$. We say that
$w$ is {\it hidden} in the text $\Xi^n$. We do not put any constraints on the 
gaps $i_{j+1}-i_j$, so in language of \cite{fsv06} this is known as
the {\it unconstrained} hidden pattern matching. The most interesting quantity of
such a problem is the number of subsequence occurrences in the text
generated by a random source. In this paper, we study the
limiting distribution of this quantity when $m$, 
the length of the pattern, grows with $n$.

Hereafter, we assume that a memoryless source generates the text $\Xi$, that is,
all symbols are generated independently with probability $p_a$ for symbol
$a\in \cA$, where the alphabet $\cA$ is assumed to be finite.
We denote by $p_w=\prod_j p_{w_j}$ the probability of the pattern $w$.
Our goal is to understand the probabilistic behavior, in particular,
the limiting distribution of the number of subsequence occurrences
that we denote by $Z:=Z_\Xi(w)$. It is known that the behavior of $Z$
depends on the order of magnitude of the pattern length $m$. 
For example, for the {\it exact pattern matching} (i.e., the pattern $w$ must
occur as a {\it string} in consecutive positions of the text), the limiting
distribution 
is normal for $m=O(1)$ (more precisely, when $n p_w\to\infty$,
hence up to $m=O(\log n)$),
but it becomes a \Polya--Aeppli distribution when $n p_w \to \lambda>0$
for some constant $\lambda$, and finally (conditioned on being non-zero)
it turns into a geometric
distribution when $np_w \to 0$  \cite{js-book}
(see also \cite{BeKo93}).  
We might expect a similar behaviour
for the subsequence pattern matching. In \cite{fsv06} it was proved 
by analytic combinatoric methods that
the number of subsequence occurrences, $Z_\Xi(w)$, is asymptotically normal
when $m=O(1)$, and not much is known beyond this regime. 
(See also \cite{BoVa02}.
Asymptotic normality for fixed $m$ follows also by
general results for $U$-statistics \cite{Hoeffding}.)
However, in many applications  -- as discussed below -- 
we need to consider patterns $w$ whose lengths  grow with $n$.
In this paper, we prove two main results.  
In Theorem~\ref{T2a} we establish that
for $m=o(n^{1/3})$ the number of subsequence occurrences is normally
distributed. Furthermore, in Theorem~\ref{Tka} we show that
under some constrains on the structure of $w$, 
the asymptotic normality can be extended to $m=o(\sqrt{n})$. 
Moreover,
for the special pattern $w=a^m$ consisting of the same symbol repeated,
we show in \refT{TLN} that for 
$m=o(\sqrt{n})$,
the distribution of number of occurrences is asymptotically normal,
while for larger $m$ (up to $cn$ for some $c>0$)
it is asymptotically log-normal. We conjecture that this 
dichotomy  is true for a large class of patterns.
Finally, for random typical $w$ we establish in Corollary~\ref{cor-random} that
$Z$ is asymptotically normal for $m=o(n^{2/5})$.

Regarding methodology, unlike \cite{fsv06} we use here probabilistic tools.
We first observe that $Z$ can be represented as a $U$-statistic
(see (\ref{eq-u}) and \refSS{Sdec}).
This suggests to apply the \citet{Hoeffding}
projection method to prove asymptotic normality of $Z$ for some large patterns.
Indeed, we first decompose $Z$ into a sum of orthogonal random variables 
with variances of decreasing order in $n$ (for $m$ not too large),
and show that the variable of
the largest variance converges to a normal distribution, proving our 
main results Theorems~\ref{T2a} and \ref{Tka}.

%applications

The hidden pattern matching problem, especially for large patterns,
finds many applications from intrusion detection,
to trace reconstruction, to deletion channel, to DNA-based storage systems
\cite{ags03,olgica19,dsv12,dg06,js-book,Mitz}. Here we discuss below in some
detail
two of them, namely the deletion channel and the trace reconstruction problem.

A deletion channel \cite{dsv12,dg06,dobrushin67,kms10,Mitz,vtr11}
with parameter $d$ takes a binary sequence
$\Xi^n=\xi_1\cdots \xi_n$ where $\xi_i \in \cA$ as input and
deletes each symbol in the sequence independently with probability $d$.
The output of such a channel is then a \emph{subsequence}
$\zeta = \zeta(x) = \xi_{i_1}...\xi_{i_M}$ of $\Xi$, where $M$ follows the binomial
distribution ${\rm Binom}(n,(1-d))$, and the indices
$i_1, ..., i_M$ correspond to the bits that are {\it not} deleted.
Despite significant effort \cite{dg06,kms10,KanMont,Mitz,vtr11}
the mutual information between the input and output of the deletion
channel and its capacity are still unknown. 
%We hope to provide a more detailed characterization of the mutual information for
%memoryless sources using results of this and forthcoming papers.
However, it turns out that the mutual information
$I(\Xi^n;\zeta(\Xi^n))$ can be exactly formulated as the problem of the
subsequence pattern matching. In \cite{dsv12} it was proved that 
\begin{align}
I(\Xi^n;\!\zeta(\Xi^n)) \!=\!\! \sum_w d^{n-|w|}\!(1-d)^{|w|}\!
\bigl(& \mathbb{E}[Z_{\Xi^n}(w) \!\log Z_{\Xi^n}\!(w)]
\notag\\ &%\hskip 1em
-  \mathbb{E}[Z_{\Xi^n}(w)] \log \mathbb{E}[Z_{\Xi^n}(w)] \bigr),
\end{align}
where the sum is over all binary sequences of length smaller than $n$ and
$Z_{\Xi^n}(w)$ is the number of subsequence occurrences of $w$ in the text $\Xi^n$.
As one can see, to find precise asymptotics of the mutual information we 
need to understand the probabilistic behavior of $Z$ for $m\le n$ and 
typical $w$.
% which is our long term goal. 
The trace reconstruction problem \cite{olgica19,holden,mcgregor14,peres17}
is related to the deletion channel problem since
we are asking how many copies of the output deletion channel we need to 
see until we can reconstruct the input sequence with high probability.
%We easily can see that $N_n =\Theta(\log n)$
%since we need at least as many copies 
%of length $n(1-d)$ as to cover the whole input sequence. However, to obtain
%more precise bounds we need a better understanding of the number of subsequence 
%occurrences $Z_{\Xi}(w)$ for large patterns $w$, which is the main focus 
%of this paper. 

\section{Main Results}
\label{sec-main}

In this section we formulate precisely our problem and present our
main results. Proofs are delayed till the next section.

\subsection{Problem formulation and notation}\label{Snot}

We consider a random string $\Xi^n=\xi_1\dots \xi_n$ of length $n$.  
We assume that $\xi_1,\xi_2,\dots$ are \iid{} random letters from a finite
alphabet $\cA$; each letter $\xi_i$ has the distribution
\begin{equation}
  \P(\xi_i=a)=p_a, \qquad a\in\cA,
\end{equation}
for some given vector $\vp=(p_a)_{a\in\cA}$; we assume $p_a>0$ for each $a$.
We may also use $\xi$ for a random letter with this distribution.

Let $w=w_1\dotsm w_m$ be a fixed string of length $m$ over the same
alphabet $\cA$. We assume $n\ge m$. 
Let
\begin{equation}
  p_w:=\prodjm p_{w_j},
\end{equation}
which is the probability that $\xi_1\dotsm \xi_m$ equals $w$.

Let $\Z=\Z_{n,w}(\xi_1\dotsm \xi_n)$ 
be the number of occurrences of $w$ as a subsequence of $\xi_1\dotsm \xi_n$.

For a set $\cS$ (in our case $\nn$ or $\mm$) and $k\ge0$,
let $\binom{\cS}k$ be the collection of sets $\ga\subseteq \cS$ with $|\ga|=k$.
Thus, $\bigabs{\binom{\cS}k}=\binom{|\cS|}{k}$.
For $k=0$, $\binom{\cS}{0}$ contains just the empty set $\emptyset$.
For $k=1$, we identify $\binom {\cS}1$ and $\cS$ in the obvious way.
We write $\ga\in\binom{\nn}{k}$ as $\set{\ga_1,\dots,\ga_k}$, where we
assume that $\ga_1<\dots<\ga_k$.
Then 
\begin{equation}\label{eq-u}
  \Z=\sum_{\ga\in\binom{\nn}m}I_\ga,
\end{equation}
where 
\begin{equation}
  I_\ga=\prodjm \indic{\xi_{\ga_j}=w_j}.
\end{equation}

\begin{remark}
  In the limit theorems, we are studying the asymptotic distribution of $\Z$.
%(or some equivalent variable). 
We then assume that \ntoo{} and (usually) \mtoo; we thus implicitly consider
a sequence of words $w^{(n)}$ of lengths $m_n=|w^{(n)}|$. 
%or perhaps sequences $n_\nu$, $w^{(\nu)}$ and $m_\nu=|w^{(\nu)}|$.
But for simplicity we do not show this in the notation. % (Usually?)
\end{remark}

We have $\E I_\ga=p_w$ for every $\ga$. Hence,
\begin{equation}\label{EZ}
\E \Z=\sum_{\ga\in\binom{\nn}m}\E I_\ga = \binom{n}{m} p_w.
\end{equation}

Further, let
\begin{equation}
  Y_\ga := p_w\qw I_\ga,
\end{equation}
so $\E Y_\ga=1$,
and
\begin{equation}\label{Zx1}
  \Zx := p_w\qw \Z = \sum_{\ga\in\binom{\nn}{m}}Y_\ga,
\end{equation}
so $\E\Zx=\binom{n}{m}$ and
\begin{equation}\label{bZx}
%  \bZx:=
\Zx-\E \Zx = p_w\qw \Z-\binom{n}{m}
=\sum_{\ga\in\binom{\nn}m}\bigpar{Y_\ga-1}.
\end{equation}

We also write $\norm{Y}_p:=\bigpar{\E |Y|^p}^{1/p}$ for the $L^p$ norm of a
random variable $Y$,
while $\norme{\vx}$ is the usual Euclidean norm of
a vector $\vx$ in some $\bbR^m$.

$C$ denotes constants that may be different at different occurrences; they
may depend on the alphabet $\cA$ and $(p_a)_{a\in \cA}$,
but not on $n$, $m$ or $w$.

Finally, $\dto$ and $\pto$ mean convergence in distribution and probability,
respectively.

We are now ready to present our main results regarding the limiting distribution
of $Z$, the number of subsequence $w=a_1, \ldots a_m$ occurrences
when $m\to \infty$. We start with a simple example,
namely,  $w=a^m=a\dotsm a$ for some $a\in \cA$, and show that
depending on whether $m=o(\sqrt{n})$ or not the number of subsequences will
follow asymptotically either the normal distribution or the log-normal
distribution.
%In our main results we prove that for $m=o(\sqrt{n}$ the number of
%subsequences $w$ is distributed asymptotically normally.  

Before we present our results we consider 
asymptotically normal and log-normal distributions in general, 
and discuss their relation.

\subsection{Asymptotic normality and log-normality}

If $X_n$ is a sequence of random variables and $a_n$ and $b_n$ are sequences
of real numbers, with $b_n>0$, then
\begin{align}
  X_n\sim\ASN(a_n,b_n)
\end{align}
means that
\begin{align}
  \frac{X_n-a_n}{\sqrt{b_n}}\dto N(0,1).
\end{align}
We say that  $X_n$ is 
\emph{asymptotically normal} if $X_n\sim\ASN(a_n,b_n)$ for some $a_n$ and
$b_n$, and 
\emph{asymptotically log-normal} if $\ln X_n\sim\ASN(a_n,b_n)$ for some
$a_n$ and $b_n$ (this assumes $X_n\ge0$).
Note that these notions are equivalent when the asymptotic variance $b_n$ is
small, as  made precise by the following lemma.

\begin{lemma}\label{Lasn}
  If\/ $b_n\to0$, and $a_n$ are arbitrary, then
  \begin{align}
    \label{lasn}
 \ln X_n\sim \ASN(a_n,b_n)
\iff %if and only if 
 X_n\sim \ASN(e^{a_n},b_ne^{2a_n}).
  \end{align}
\end{lemma}
\begin{proof}
  By replacing $X_n$ by $X_n/e^{a_n}$, we may assume that $a_n=0$.
If $\ln X_n\sim\ASN(0,b_n)$ with $b_n\to0$, then $\ln X_n\pto0$, and thus 
$X_n\pto 1$. It follows that $\ln X_n/(X_n-1)\pto1$ (with $0/0:=1$), and
thus 
\begin{align}
  \frac{X_n-1}{b_n\qq}
=
  \frac{X_n-1}{\ln X_n}
  \frac{\ln X_n}{b_n\qq}
\dto N(0,1),
\end{align}
and thus $X_n\sim\ASN(1,b_n)$.

The converse is proved by the same argument.
\end{proof}

\begin{remark}\label{RASN}
  \refL{Lasn} is best possible.
Suppose that $\ln X_n\sim\ASN(a_n,b_n)$. 
If $b_n\to b>0$, 
then $\ln\bigpar{X_n/e^{a_n}}=\ln X_n-a_n\dto N(0,b)$, and thus
\begin{align}
  X_n/e^{a_n} \dto e^{\zeta_b}, \qquad \zeta_b\sim N(0,b).
\end{align}
In this case (and only in this case), 
$X_n$ thus converges in distribution, after scaling, 
to a log-normal distribution.
If $b_n\to\infty$, then no linear scaling of $X_n$ can
converge in distribution to a non-degenerate limit, as is easily seen.
\end{remark}

\subsection{A simple example}\label{Saaa}

We consider first a simple example where the asymptotic distribution can be
found easily by explicit calculations.
Fix $a\in\cA$ and let $w=a^m=a\dotsm a$, a string with $m$ identical
letters.
Then, if $N=N_a$ is the number of occurrences of $a$ in $\xin$, then
\begin{align}\label{emma}
  \Z=\binom{\Na}{m}.
\end{align}
We will show that $\Z$ is asymptotically normal if $m$ is small, and
log-normal for larger $m$.

\begin{theorem}\label{TLN}
Let $w=a^m$.
Suppose that $m<np_a$, with $np_a-m\gg n\qq$.
  \begin{romenumerate}
  \item \label{TLN1}
Then
\begin{align}\label{tln1}
  \ln\Z\sim
\ASN
\Bigpar{\ln\binom{np_a}{m},\,n\Bigabs{\ln\Bigpar{1-\frac{m}{np_a}}}^2p_a(1-p_a)}.
\end{align}
\item \label{TLN2}
In particular, if $m=o(n)$, then 
\begin{align}\label{tln2}
  \ln\Z\sim
\ASN\Bigpar{\ln\binom{np_a}{m},\bigpar{p_a\qw-1}\frac{m^2}{n}}.
\end{align}
\item \label{TLN3}
If $m=o\bigpar{n\qq}$, then this implies
\begin{align}\label{tln3}
  \Z/\E\Z\sim
\ASN\Bigpar{1,\bigpar{p_a\qw-1}\frac{m^2}{n}},
\end{align}
and thus
\begin{align}\label{tln4}
  \Z\sim
\ASN\Bigpar{\E\Z,\bigpar{p_a\qw-1}\frac{m^2}{n}(\E\Z)^2}
.\end{align}
  \end{romenumerate}
\end{theorem}

\begin{proof}
\pfitemref{TLN1}
We have $\Na\sim\Bin(n,\pa)$. Define
$%\begin{align}
  Y:=\Na-n\pa
$. %\end{align}
Then, by the Central Limit Theorem, 
\begin{align}\label{hw}
  Y\sim\ASN\bigpar{0,n\pa(1-\pa)}.
\end{align}
By \eqref{emma}, we have
\begin{align}\label{jesper}
  \ln\Z-\ln\binom{n\pa}{m}
&= \ln\binom{n\pa+Y}{m}-\ln\binom{n\pa}{m}
\notag\\&
=\ln\gG(n\pa+Y+1)-\ln\gG(n\pa+Y-m+1)-\ln m!
\notag\\&\qquad
-\bigpar{\ln\gG(n\pa+1)-\ln\gG(n\pa-m+1)-\ln m!}
\notag\\&
=\int_{y=0}^Y\int_{x=-m}^0(\ln\gG)''(n\pa+x+y+1)\dd x\dd y 
\end{align}
where $\Gamma(x)$ is the Euler gamma function.
We fix
a sequence $\go_n\to\infty$ such
that
$n\pa-m\gg\go_n\gg n\qq$; this is possible by the assumption. 
Note that \eqref{hw} implies that $Y/\go_n\pto0$, and thus
$\P(|Y|\le\go_n)\to1$. We may thus in the sequel assume $|Y|\le\go_n$.
We assume also that $n$ is so large that $n\pa-m\ge 2\go_n>0$.

Stirling's formula implies, by taking the logarithm and differentiating
twice (in the complex half-plane $\Re z>\frac12$, say)
\begin{align}\label{sti''}
  (\ln\gG)''(x) = \frac{1}{x}+O\Bigparfrac{1}{x^2}% +O\bigpar{x\qww}
=\frac{1}{x}\Bigpar{1+O\Bigparfrac{1}{x}}
,\qquad x\ge1.
\end{align}
Consequently, \eqref{jesper} yields, noting the assumptions just made imply
$|Y|\le\go_n\le \frac12(n\pa-m)$,
\begin{align}\label{jeppe}
  \ln\Z-\ln\binom{n\pa}{m}
&
=\int_{y=0}^Y\int_{x=-m}^0
\frac{1}{n\pa+x+y+1} \Bigpar{1+O\Bigparfrac{1}{n\pa-m}}\dd x\dd y
\notag\\&
=\int_{y=0}^Y\int_{x=-m}^0
 \frac{1}{n\pa+x}\Bigpar{1+O\Bigparfrac{\go_n}{n\pa-m}}\dd x\dd y
\notag\\&
=\Bigpar{1+O\Bigparfrac{\go_n}{n\pa-m}} 
Y\int_{x=-m}^0 \frac{1}{n\pa+x}\dd x
\notag\\&
%=\bigpar{1+o(1)}
%Y \bigpar{\ln \xpar{n\pa}-\ln \xpar{n\pa-m}}
%\notag\\&
=\bigpar{1+o(1)}
Y {\ln \frac{n\pa}{n\pa-m}}
.\end{align}
Consequently, using also \eqref{hw}, we obtain
\begin{align}
\frac{ \ln\Z-\ln\binom{n\pa}{m}}{n\qq\bigabs{\ln\bigpar{1-\frac{m}{n\pa}}}}
=\bigpar{1+\op(1)} \frac{Y}{n\qq} \dto N\bigpar{0,\pa(1-\pa)},
\end{align}
which is equivalent to \eqref{tln1}.

\pfitemref{TLN2}
If $m=o(n)$, then $\bigabs{\ln\bigpar{1-\frac{m}{n\pa}}}\sim
\frac{m}{n\pa}$, and \eqref{tln2} follows.

\pfitemref{TLN3}
If $m=o(n\qq)$, then 
\ref{TLN2} applies, so \eqref{tln2} holds; hence \refL{Lasn} implies
\begin{align}\label{t4}
  Z\bigm/\binom{n\pa}{m}\sim
\ASN \Bigpar{1,\bigpar{p_a\qw-1}\frac{m^2}{n}}.
\end{align}
Furthermore, 
\begin{align}
  \E Z = \binom{n}{m}\pa^m
=\frac{n^m e^{O(m^2/n)}}{m!}\pa^m
\sim \frac{n^m}{m!}\pa^m
\end{align}
and, similarly, $\binom{n\pa}{m}\sim\frac{n^m\pa^m}{m!}$.
Hence, $\E Z \sim \binom{n\pa}{m}$ and \eqref{tln3} follows from \eqref{t4};
\eqref{tln4} is an immediate consequence.
\end{proof}

\begin{example}\label{Eaaa}
  Let $w=a^m$ as in \refT{TLN}, and let 
$m\sim c\sqrt n$ for some $c>0$.
Then, as \ntoo,
by \refT{TLN}\ref{TLN2}, with $Z=Z_n$,  $z_n:=\binom{np_a}{m}$ and
$\gss:=c^2(p_a-1)$, 
\begin{align}
  \ln Z_n \sim \ASN\bigpar{\ln z_n,\gss}
\end{align}
and thus
\begin{align}
  \ln\frac{Z_n}{z_n} \dto N\bigpar{0,\gss}.
\end{align}
Hence, $Z_n/z_n$ converges in distribution to a log-normal distribution,
so $Z_n$ is
asymptotically log-normal but not
asymptotically normal.
See also \refR{RASN}.
\end{example}

 \subsection{General results}

We now present our main results.
However, first we discuss the road map of our approach.
First, we observe that the representation  \eqref{eq-u} shows that
$Z$ can be viewed as a $U$-statistic.
For convenience, we consider $\Zx$ in \eqref{Zx1},
which differs from $\Z$ by a constant factor only,
and  show in \eqref{Nb1} that $\Zx - \E\Zx$
can be
decomposed into 
a sum $\sum_{\ell=1}^m V_\ell$ of orthogonal random variables
$V_\ell$ such that, when $m$ is not too large, 
$\Var \bigpar{\sum_{\ell=2}^m V_{\ell}} =o(\Var V_1)$.
Next, in 
Lemma~\ref{LV} we prove that $V_1$ appropriately normalized converges
to the standard normal distribution. This will allow us to conclude the
asymptotic normality of $Z$.

In this paper, we only consider the region $m=o\bigpar{n\qq}$.
First, for $m=o\bigpar{n\qqq}$ we claim that the number of subsequence
occurrences always
is asymptotically normal. 
%More precisely, our first main result is as follows.
%The proof is given in the next section.

\begin{theorem}\label{T2a}
  If\/ $m=o\bigpar{n\qqq}$, then
%the number of $w$ occurrences as a subsequence is asymptotically normal: 
\begin{align}\label{t1Na}
%  \frac{\Z-\E \Z}{p_w\sigma_1}  \dto N(0,1),
Z \sim \ASN\Bigpar{\binom{n}{m} p_w,\gss_1p_w^2},
\end{align}
where 
\begin{align}
\sigma_1^2 &= \sum_{i=1}^n \sum_{a\in \cA} p_a^{-1} \left(\sum_{j:~w_j=a} 
\binom{i-1}{j-1}\binom{n-i}{m-j}\right)^2 -n \binom{n-1}{m-1}^2.  
\label{ws3}
\end{align}
Furthermore,
$\E Z = \binom{n}{m} p_w$ and\/
$\Var Z \sim p_w^2\sigma_1^2$. 
\end{theorem}

In the second main result, we restrict the patterns $w$ to such that are not
typical for the random text; however, we will allow
$m=o\bigpar{n\qq}$.  
%More precisely, we claim the following result.
% which is proved in the next section.

\begin{theorem}\label{Tka}
Let\/ $\vq=(q_a)_{a\in\cA}$ be the proportions of the letters in $w$, \ie,
$q_a:=\frac{1}{m}\sumjm\indic{w_j=a}$.
Suppose that  $\liminf_\ntoo\norme{\vq-\vp}>0$.
If further\/ $m=o\bigpar{n\qq}$, then we have the asymptotic normality
\begin{align}\label{t1Na-2}
%  \frac{\Z-\E \Z}{p_w\sigma_1}  \dto N(0,1),
Z \sim \ASN\Bigpar{\binom{n}{m} p_w,\gss_1p_w^2},
\end{align}
where $\gss_1$ is given by \eqref{ws3}.
Furthermore,
$\E Z = \binom{n}{m} p_w$ and\/
$\Var Z \sim p_w^2\sigma_1^2$. 
\end{theorem}

\section{Analysis and Proofs}

In this section we will prove our main results. We start with some preliminaries.

\subsection{Preliminaries and more notation}

Let, for $a\in\cA$,
\begin{equation}\label{gfa}
  \gf_a(x):=p_a\qw\indic{x=a}-1.
\end{equation}
Thus, letting $\xi$ be any random variable with the distribution of $\xi_i$,
\begin{align}
  \label{gf0}
\E\gf_a(\xi)=0,
\qquad a\in\cA.
\end{align}
Let $\px:=\min_a p_a$ and 
\begin{equation}\label{B}
B:=\px\qw-1.  
\end{equation}

\begin{lemma}\label{Lgf}
  Let $\gf_a$ and $B$ be as above. %in \eqref{gfa}  and \eqref{B}.
\begin{romenumerate}
\item\label{Lgf1} 
For every $a\in\cA$,
  \begin{align}\label{lgf1}
    \E \bigsqpar{\gf_a(\xi)^2} = p_a\qw-1 \le B.
  \end{align}
\item\label{Lgf2}
%Let $c_1:=\min_{a\in\cA} \bigpar{p_a\qw-1}\qq>0$. 
 For some $c_1>0$ and every $a\in\cA$,
  \begin{align}
    \norm{\gf_a(\xi)}_2 = \bigpar{p_a\qw-1}\qq \ge c_1.\label{lgf2}
  \end{align}
\item\label{Lgf3} 
For any vector $\vr=(r_a)_{a\in\cA}$ with $\sum_ar_a=1$,
  \begin{align}\label{lgf3}
    \Bignormm{\sum_{a\in\cA}r_a\gf_a(\xi)}
\ge \norme{\vr-\vp}
:=\Bigpar{\sum_{a\in\cA}|r_a-p_a|^2}\qq.
  \end{align}
\end{romenumerate}
\end{lemma}

\begin{proof}
The definition \eqref{gfa} yields
\begin{equation}\label{lgf1p}
  \E \bigsqpar{\gf_a(\xi)^2}=
p_a^{-2}\Var\bigsqpar{\indic{\xi=a}}=p_a^{-2}p_a(1-p_a)=p_a\qw-1.
\end{equation}
Hence, \eqref{lgf1} and \eqref{lgf2} follow, with $B$ given by \eqref{B}.

Finally, for every $x\in\cA$, by \eqref{gfa} again,
\begin{align}
\sum_{a\in\cA}r_a\gf_a(x)
=r_xp_x\qw-\sum_{a\in\cA}r_a
=r_x/p_x-1
\end{align}
and thus
\begin{align}
  \E\Bigpar{\sum_{a\in\cA}r_a\gf_a(\xi)}^2
=\sum_{a\in\cA} p_a \bigpar{r_a/p_a-1}^2
=\sum_{a\in\cA} p_a\qw \bigpar{r_a-p_a}^2
\end{align}
and \eqref{lgf3} follows.
\end{proof}

\subsection{A decomposition}\label{Sdec}

The representation  \eqref{eq-u} shows that
$Z$ is a special case of a $U$-statistic. 
(Recall that, in general, a $U$-statistic is a sum over subsets $\ga$ as in
\eqref{eq-u} of $f\bigpar{\xi_{\ga_1},\dots,\xi_{\ga_k}}$
for some function $f$.)
For fixed $m$, the general theory of \citet{Hoeffding} applies
and yields asymptotic normality.
(Cf.\ \cite[Section 4]{SJ287} for a related problem.)
For $m\to \infty$ (our main interest), we can still use the 
orthogonal decomposition
of \cite{Hoeffding}, which in our case takes the following form.

By the definitions in \refS{Snot} and \eqref{gfa}, 
\begin{equation}
  Y_\ga 
= \prodjm \bigpar{p_{w_j}\qw\indic{\xi_{\ga_j}=w_j}}
= \prodjm\bigpar{\gf_{w_j}(\xi_{\ga_j})+1}.
\end{equation}
By multiplying out
this product, we obtain
\begin{equation}
  Y_\ga 
= \sum_{\gam\subseteq\mm}\prod_{j\in\gam}\gf_{w_j}(\xi_{\ga_j}).
\end{equation}
Hence,
\begin{equation}\label{Zx2}
  \Zx = \sum_{\ga\in\binom{\nn}m} Y_\ga
= \sum_{\ga\in\binom{\nn}m}
\sum_{\gam\subseteq\mm}\prod_{j\in\gam}\gf_{w_j}(\xi_{\ga_j})
= \sum_{\ga\in\binom{\nn}m} \sum_{\gam\subseteq\mm}
 \prod_{k=1}^{|\gam|}\gf_{w_{\gam_k}}(\xi_{\ga_{\gam_k}}).
\end{equation}
We rearrange this sum. First, let $\ell:=|\gam|\in\mm$, and consider all
terms with a given $\ell$. For each $\ga$ and $\gam$, with $|\gam|=\ell$,
let
\begin{equation}
  \ga_\gam:=\set{\ga_{\gam_1},\dots,\ga_{\gam_\ell}} \in \binom{\nn}{\ell}.
\end{equation}
For given $\gam\in\binom{\mm}\ell$ and $\gb\in\binom{\nn}\ell$, 
the number of $\ga\in\binom{\nn}{m}$ such that $\ga_\gam=\gb$ 
equals the number of ways to choose, for each $k\in[\ell+1]$, 
$\gam_k-\gam_{k-1}-1$ elements of $\ga$ in a gap of length
$\gb_k-\gb_{k-1}-1$,
where we define $\gb_0=\gam_0=0$ and $\gb_{\ell+1}=n+1$, $\gam_{\ell+1}=m+1$;
this number is %thus
\begin{equation}\label{cgb}
\cbg
:=
\prod_{k=1}^{\ell+1}\binom{\gb_{k}-\gb_{k-1}-1}{\gam_{k}-\gam_{k-1}-1}.
\end{equation}
Consequently, combining the terms in \eqref{Zx2} with the same $\ga_\gam$,
\begin{equation}\label{Nb0}
  \Zx 
= 
\sum_{\ell=0}^m
\sum_{\gam\in\binom{\mm}\ell}
\sum_{\gb\in\binom{\nn}\ell} 
\cbg
 \prod_{k=1}^{\ell}\gf_{w_{\gam_k}}(\xi_{\gb_{k}})
.%=:\sum_{\ell=0}^m V_\ell,
\end{equation}
We define, for $0\le \ell\le m$ and $\gb\in\binom{\nn}\ell$,
\begin{equation}\label{Vlgb}
  V_{\ell,\gb}
:= 
\sum_{\gam\in\binom{\mm}\ell}
\cbg
 \prod_{k=1}^{\ell}\gf_{w_{\gam_k}}(\xi_{\gb_{k}})
\end{equation}
and 
\begin{equation}\label{Vl}
V_\ell:=
\sum_{\gb\in\binom{\nn}\ell} V_{\ell,\gb}.
\end{equation}
Thus \eqref{Nb0} yields the decomposition
\begin{align}\label{Nb1}
  \Zx=
\sum_{\ell=0}^m V_\ell.
\end{align}
For $\ell=0$, $\binom{\nn}{0}$ contains only the empty set $\emptyset$,
and
\begin{align}
V_0=V_{0,\emptyset}=\binom{n}{m}=\E \Zx.  
\end{align}
Furthermore, note that two summands in \eqref{Nb0} with different $\gb$ are
orthogonal, as a consequence of \eqref{gf0} and independence of different
$\xi_i$. 
Consequently, the variables $V_{\ell,\gb}$ 
($\ell\in\mm$, $\gb\in\binom{\nn}{\ell}$)  are orthogonal, and hence
the variables $V_\ell$ ($\ell=0,\dots,m$) are orthogonal.

Let
\begin{equation}\label{vasa}
  \gss_\ell:=\Var(V_\ell)=\E V_\ell^2=\sum_{\gb\in\binom{\nn}\ell}\E V_{\ell,\gb}^2,
\qquad 1\le \ell\le m.
\end{equation}
%For convenience, let $\gss_\ell:=0$ for $\ell>m$.

 Note also that
by the combinatorial definition of $\cbg$ given before
\eqref{cgb}, we see that
\begin{align}\label{sumc2}
\sum_{\gb\in\binom{\nn}\ell}\cbg
 =\binom{n}{m}, 
\end{align}
since this is just the number of $\ga\in\binom{\nn}{m}$, and
\begin{align}\label{sumc}
\sum_{\gam\in\binom{\mm}\ell}\cbg
 =\binom{n-\ell}{m-\ell}, 
\end{align}
since this sum
is the total number of ways to choose
$m-\ell$ elements of the $n-\ell$ elements of $\ga$ in the gaps.

\subsection{The projection method}
We use the projection method used by \citet{Hoeffding} to prove asymptotic
normality for $U$-statistics. 
Translated to the present setting, 
the idea of the projection method 
is to approximate $\Zx-\E\Zx=\Zx-V_0$ by $V_1$, thus ignoring all 
terms with $\ell\ge2$ in the sum in \eqref{Nb1}.
In order to do this, we estimate variances.

First, by \eqref{lgf1} and
%\begin{equation}\label{vb}
%  \E \bigpar{\gf_a(\xi)^2}=
%p_a^{-2}\Var\bigsqpar{\indic{\xi=a}}=p_a^{-2}p_a(1-p_a)=p_a\qw-1\le B.
%\end{equation}
the independence of the $\xi_i$,
\begin{equation}\label{v1}
\Bignorm{\prod_{k=1}^{\ell}\gf_{w_{\gam_k}}(\xi_{\gb_{k}})}_2
=
\Bigpar{ \prod_{k=1}^{\ell}\E\bigabs{\gf_{w_{\gam_k}}(\xi_{\gb_{k}})}^2}\qq
\le B^{\ell/2}.
\end{equation}
By Minkowski's inequality, \eqref{Vlgb}, \eqref{v1} and \eqref{sumc},
\begin{align}\label{v2w0}
\bignorm{ V_{\ell,\gb}}_2
&\le 
\sum_{\gam\in\binom{\mm}\ell}\cbg B^{\ell/2}
%\notag\\&
=B^{\ell/2}\binom{n-\ell}{m-\ell}
\end{align}
or, equivalently,
\begin{align}\label{v2w}
\E V_{\ell,\gb}^2
\le
B^{\ell}\binom{n-\ell}{m-\ell}^2.
\end{align}
This leads to the following estimates.

\begin{lemma}
  \label{L0}
For $1\le\ell\le m$,
\begin{align}\label{l0}
\gss_\ell:=\E V_\ell^2& 
%\sum_{\gb\in\binom{\nn}\ell}
%\E V_{\ell,\gb}^2
\le
\gxx_\ell:= 
B^\ell \binom{n}{\ell}\binom{n-\ell}{m-\ell}^2.
\end{align}
\end{lemma}

\begin{proof}
 The definition of $V_\ell$ in \eqref{Vl} and \eqref{v2w} yield, since
the summands $V_{\ell,\gb}$ are orthogonal,
\begin{align}
\gss_\ell:=\E V_\ell^2& =
\sum_{\gb\in\binom{\nn}\ell}
\E V_{\ell,\gb}^2
%\notag\\&
\le 
\binom{n}{\ell}B^\ell\binom{n-\ell}{m-\ell}^2 ,
\end{align}
as needed.
\end{proof}

Note that, for $1\le\ell<m$,
\begin{align}\label{v2}
  \frac{\gxx_{\ell+1}}{\gxx_\ell}
=B\frac{\binom{n}{\ell+1}\binom{n-\ell-1}{m-\ell-1}^2}
{\binom{n}{\ell}\binom{n-\ell}{m-\ell}^2}
=B\frac{n-\ell}{\ell+1}\Bigparfrac{m-\ell}{n-\ell}^2
\le B\frac{m^2}{(\ell+1) n}.
\end{align}

\begin{lemma}\label{L1}
  If $m\le B\qqw n\qq$, then
\begin{align}
  \Var\bigpar{\Zx-V_1}
\le B^2 m^2\binom{n-1}{m-1}^2.
\end{align}
\end{lemma}
\begin{proof}
  By \eqref{v2} and the assumption,
for $1\le\ell<m$,
\begin{align}\label{v22}
  \frac{\gxx_{\ell+1}}{\gxx_\ell}
\le \frac{1}{\ell+1} \le \frac12,
\end{align}
and thus, summing a geometric series,
\begin{align}
  \Var\bigpar{\Zx-V_1}
&=\sum_{\ell=2}^m \Var\bigpar{V_\ell}
\le\sum_{\ell=2}^m \gxx_\ell
\le\sum_{\ell=2}^m 2^{2-\ell}\gxx_2
\le2\gxx_2
\notag\\&
=B^2 n(n-1)\binom{n-2}{m-2}^2
\le B^2 m^2\binom{n-1}{m-1}^2.
\end{align}
\end{proof}

\subsection{The first term $V_1$}\label{Sfirst}

For $\ell=1$, we identify $\binom{\nn}\ell$ and $\nn$, and we write
$\VV{i}:=V_{1,\set i}$. Note that, by \eqref{cgb},
\begin{align}\label{cij}
  c(i,j):=c\bigpar{\set i,\set j} =\binom{i-1}{j-1}\binom{n-i}{m-j}.
\end{align}

\begin{remark}\label{RHG}
For later use, we  define also
  \begin{align}\label{pi}
    \pi(i,j):=\frac{c(i,j)}{c(1,1)}=\frac{c(i,j)}{\binom{n-1}{m-1}}.
  \end{align}
Then, for fixed $i$, $(\pi(i,j))_j$ is a (shifted) hypergeometric
distribution:
\begin{align}\label{hg}
\pi(i,j)=\P(X=j-1)= \frac{\binom{i-1}{j-1} \binom{n-i}{m-j}}{\binom{n-1}{m-1}}
%\qquad X\sim \HG\bigpar{n-1,m-1,i-1}.  
\end{align}
which we write as
\begin{equation}
X\sim \HG\bigpar{n-1,m-1,i-1}.  
\end{equation}
\end{remark}

For $\ell=1$, 
\eqref{Vl} and \eqref{Vlgb} become
\begin{align}\label{w1}
  V_1=\sumin\VV i
\end{align}
with, using \eqref{cij},
\begin{align}\label{vvi}
  \VV i
=\sumjm c(i,j)\gf_{w_j}(\xi_i)
=\sumjm\binom{i-1}{j-1}\binom{n-i}{m-j}\gf_{w_j}(\xi_i).
\end{align}
Note that $\VVi$ is a function of $\xi_i$, and thus the random variables
$\VVi$ are independent. Furthermore, \eqref{gf0} implies $\E\VVi=0$. 
Let 
\begin{align}\label{tauu0}
\tauu_i:=\Var \VVi=\E\VVi^2.  
\end{align}
Then, see \eqref{vasa},
\begin{align}\label{tauu}
\gssw=
  \Var V_1 = \sumin \Var \VVi = \sumin \tauu_i.
\end{align}
Observe that it follows from \eqref{vvi} and \eqref{gfa} that 
\begin{equation}
\label{ws1}
\tauu_i=\sum_{a\in \cA} p_a^{-1} \left(\sum_{j:~w_j=a} \binom{i-1}{j-1}
\binom{n-i}{m-j}\right)^2
-\binom{n-1}{m-1}^2.
\end{equation}

Taking $\ell=1$ in \eqref{v2w} yields the upper bound
\begin{align}
  \label{tauuibound}
\tauu_i =\E \VVi^2\le B\binom{n-1}{m-1}^2,
\qquad i\in\nn.
\end{align}
Summing over $i$, or using \eqref{l0}, we obtain
\begin{align}\label{lol}
\gss_1:=\E V_1^2& 
\le
\gxx_1:= 
B n\binom{n-1}{m-1}^2.
\end{align}

\begin{remark}
 The upper bound \eqref{lol}, which is the case $\ell=1$ of \refL{L0},  
is achievable. Indeed, for $w=a \cdots a$,
by \eqref{ws1},
\begin{align}\label{cectau}
  \tau_i^2=(p_a^{-1}-1)\binom{n-1}{m-1}^2,
  \end{align}
and thus by \eqref{tauu},
  \begin{align}
\sigma_1^2=n (p_a^{-1}-1)\binom{n-1}{m-1}^2.
\label{cec}
\end{align}
Now choose $a$ to minimize $p_a$ and recall \eqref{B}.

We will see in \refL{Lk} that
the bound \eqref{lol} 
is sharp within a constant factor much more generally.
\end{remark}

We show also a general lower bound.
This too is sharp, see Section~\ref{sec-low}.

\begin{lemma}\label{L2}
  There exists $c,c'>0$ such that 
  \begin{align}\label{l2}
    \gss_1\ge \frac{c}{m}\gxx_1 = c'\frac{n}{m}\binom{n-1}{m-1}^2.
  \end{align}
\end{lemma}

\begin{proof}
We consider the first term in the sum in \eqref{vvi} separately, and write
\begin{align}\label{magnus}
  \VVi=c(i,1)\gf_{w_1}(\xi_i)+\VVxi,
\end{align}
where
\begin{align}\label{vvxi}
  \VVxi:=
 \sum_{j=2}^mc(i,j)\gf_{w_j}(\xi_i).
\end{align}
We have, by \eqref{cij}, $c(i,1)=\binom{n-i}{m-1}$. Consequently, for any
$i\in\nn$, 
\begin{align}\label{winston}
  \frac{c(i,1)}{c(1,1)}
&=\frac{\binom{n-i}{m-1}}{\binom{n-1}{m-1}}
=\frac{\prod_{k=0}^{m-2}(n-i-k)}{\prod_{k=0}^{m-2}(n-1-k)}
=\prod_{k=0}^{m-2}\Bigpar{1-\frac{i-1}{n-1-k}}
\notag\\&
\ge 1-\sum_{k=0}^{m-2}\frac{i-1}{n-1-k}
\ge 1-\frac{m(i-1)}{n-m+1}.
\end{align}

Let $\gd\le1/4$ be a fixed small positive number, chosen later.
Assume that $i\le 1+\gd n/m$. In particular, 
either $i=1$ or $m\le m(i-1)\le\gd n < n/2$,
and thus \eqref{winston} implies
\begin{align}\label{anna}
  \frac{c(i,1)}{c(1,1)}
%\ge 1-\sum_{k=0}^{m-2}\frac{i-1}{n-1-k}
\ge 1-\frac{m(i-1)}{n-m}
\ge 1-\frac{\gd n}{n/2}=1-2\gd.
\end{align}
By \eqref{sumc}, \eqref{anna} implies
\begin{align}
  \sum_{j=2}^mc(i,j)=\binom{n-1}{m-1}-c(i,1)=c(1,1)-c(i,1)
\le 2\gd c(1,1).
\end{align}
Hence, by \eqref{vvxi}, Minkowski's inequality and \eqref{lgf1}, 
\cf{} \eqref{v2w0},
\begin{align}\label{gabriel}
\bignorm{\VVxi}_2
&\le 
\sum_{j=2}^m c(i,j)\bignorm{\gf_{w_j}(\xi_i)}_2
\le \sum_{j=2}^m c(i,j) B^{1/2}
%\notag\\&
\le 2\gd B^{1/2} c(1,1).
\end{align}
%Let $c_1:=\min_{a\in\cA} \bigpar{p_a\qw-1}\qq>0$. 
Furthermore,  \eqref{lgf2} and
\eqref{anna} yield
\begin{align}\label{olof}
  \bignorm{c(i,1)\gf_{w_1}(\xi_i)}_2
\ge c(i,1) c_1 \ge c_1(1-2\gd)c(1,1)\ge \tfrac12 c_1c(1,1).
\end{align}
Finally, \eqref{magnus} and the triangle inequality yield, using
\eqref{olof} and \eqref{gabriel},
\begin{align}
\bignorm{\VVi}_2\ge \bignorm{c(i,1)\gf_{w_1}(\xi_i)}_2-\bignorm{\VVxi}_2
\ge \bigpar{\tfrac12 c_1-2\gd B\qq}c(1,1).
\end{align}
We now choose $\gd:= c_1/(8B\qq)$, and find that for some $c_2>0$,
\begin{align}
  \tauu_i:=\bignorm{\VVi}_2^2\ge c_2c(1,1)^2,
\qquad i\le 1+\gd n/m.
\end{align}
Consequently, by \eqref{tauu},
\begin{align}
  \gss_1 =\sumin\tauu_i \ge \frac{\gd n}{m} c_2 c(1,1)^2
  =c_3\frac{n}{m}\binom{n-1}{m-1}^2. 
\end{align}
This proves \eqref{l2}, with $c':=c_3$ and $c=c'/B$.
\end{proof}

\begin{lemma}\label{LV}
  Suppose that $m=o(n)$. Then $V_1$ is asymptotically normal:
  \begin{align}\label{lw}
    V_1/\gs_1 \dto N(0,1).
  \end{align}
\end{lemma}

\begin{proof}
We  show that the central limit theorem applies to the sum
$V_1=\sum_i \VVi$ in \eqref{w1}. The terms $\VVi$ are independent and have
means $\E \VVi=0$. We verify Lyapunov's condition.

%Since the random variable $\xi$ takes values in the finite set $\cA$, the
%  linear space $\cV$ of functions of $\xi$ has finite dimension $|\cA|$. 
The random variable $\xi$ is defined on some probability space $(\gO,\cF,P)$
and takes values in the finite set $\cA$. Thus the
linear space $\cV$ of functions $\gO\to\bbR$ of 
the form $f(\xi)$ has finite dimension $|\cA|$.
Moreover, every function in $\cV$ is bounded.
The $L^2$ and $L^3$ norms $\norm\cdot_2$ and $\norm\cdot_3$ are thus finite
on $\cV$, and are thus both norms
on the finite-dimensional vector space $\cV$; hence there exists a constant
$C$ such that for any function $f$,
\begin{align}
  \norm{f(\xi)}_3\le C \norm{f(\xi)}_2.
\end{align}
In particular, since the definition \eqref{vvi} shows that $\VVi$ is a function
of $\xi_i\eqd \xi$,
\begin{align}\label{Cv}
  \norm{\VVi}_3\le C \norm{\VVi}_2 = C \tau_i,
\qquad 1\le i\le n.
\end{align}

Furthermore, by \eqref{tauuibound} and \eqref{l2},
\begin{align}\label{eleonora}
  \frac{\max_i\tauu_i}{\gss_1}
\le \frac{B\binom{n-1}{m-1}^2}{c'\frac{n}{m}\binom{n-1}{m-1}^2}
=C \frac{m}{n} =o(1).
\end{align}

Consequently, using \eqref{Cv}, \eqref{tauu} and \eqref{eleonora},
\begin{align}
  \frac{\sumin \E|\VVi|^3}{\gs_1^3}
&
=\frac{\sumin \norm{\VVi}_3^3}{\gs_1^3}%{\bigpar{\sumin \tauu_i}^{3/2}}
\le \frac{C\sumin \tau_i^3}{\gs_1^3}%{\bigpar{\sumin \tauu_i}^{3/2}}
\le C\frac{\max_i\tau_i\sumin \tau_i^2}{\gs_1^3}%{\bigpar{\sumin \tauu_i}^{3/2}}
\notag\\&
= C\frac{\max_i\tau_i}{\gs_1}
=o(1).
\end{align}
This shows the Lyapunov condition, and thus a standard form of the central
limit theorem, \cite[Theorem 7.2.4 or 7.6.2]{Gut}, yields
\eqref{lw}.
\end{proof}

\subsection{Proofs of Theorem~\ref{T2a} and \ref{Tka}}

We next prove a general theorem showing asymptotic normality under some
conditions.
%To apply these conditions is another problem that will be dealt with later.

\begin{theorem}\label{T1}
  Suppose that \ntoo{} and that 
  \begin{align}\label{t1b}
m^2\binom{n-1}{m-1}^2 
%%= o\Bigpar{\sumin\tauu_i}.
= o\bigpar{\gss_1}.
  \end{align}
Then
\begin{align}\label{t1var}
\Var \Z = p_w^2\Var\Zx \sim p_w^2\gss_1
  \end{align}
and
  \begin{align}
\label{t1Zx}
  \frac{\Zx-\E \Zx}{\gsw}&\dto N(0,1),
%\end{align}
%and
%\begin{align}
\\
\label{t1N}
  \frac{\Z-\E \Z}{(\Var \Z)\qq}&=
  \frac{\Zx-\E \Zx}{(\Var \Zx)\qq}\dto N(0,1).
\end{align}
\end{theorem}

\begin{proof}
By \refL{L1} and \eqref{t1b},
\begin{align}\label{ask}
  \Var\Bigpar{\frac{\Zx-V_1}{\gsw}}
= \frac{  \Var\xpar{\Zx-V_1}}{\gssw}
\le B^2 \frac{m^2\binom{n-1}{m-1}^2}{\gssw}=o(1).
\end{align}
Hence, recalling $\E V_1=0$,
\begin{align}\label{eva}
\frac{\Zx-\E\Zx-V_1}{\gsw}
\pto0.
\end{align}
Combining \eqref{lw} and \eqref{eva}, we obtain \eqref{t1Zx}.

Furthermore, by \eqref{ask}, and since the terms in \eqref{Nb1} are orthogonal,
\begin{align}
  \Var\Zx = \Var V_1 + \Var\bigpar{\Zx-V_1} =\gssw + o(\gssw)
\sim \gssw,
\end{align}
which  yields \eqref{t1var}, and also shows that
we may replace $\gs_1$ by $(\Var\Zx)\qq$ in \eqref{t1Zx},
which yields \eqref{t1N};
the equality in \eqref{t1N} is a trivial
consequence of \eqref{Zx1}.
\end{proof}

%\section{Verifying the condition}

%Let us look at some cases.
%First we show a general, but rather weak, result which is basically our 
Now we are ready to prove our main results. % Theorem~\ref{T2a}.

%\begin{theorem}\label{T2}
%  If\/ $m=o\bigpar{n\qqq}$, then
%the asymptotic normality \eqref{t1Zx}--\eqref{t1N} holds.
%\end{theorem}

\begin{proof}[Proof of \refT{T2a}]
  By \refL{L2},
  \begin{align}
    \frac{m^2\binom{n-1}{m-1}^2}{\gss_1}
\le C \frac{m^3}{n} = o(1).
  \end{align}
Thus \eqref{t1b} holds, and the result follows by \refT{T1} together with
\eqref{EZ} and \eqref{Zx1}.
\end{proof}

Recall that in \refT{Tka},
the range of $m$ is improved, assuming that $w$ is \emph{not} typical for the
random source with probabilities $\vp=(p_a)_{a\in\cA}$ that we consider.

\begin{proof}[Proof of \refT{Tka}]
  By \refT{T1}, with \eqref{t1b} verified by \refL{Lk} below.
\end{proof}

\begin{lemma}\label{Lk}
Let $\vq=(q_a)_{a\in\cA}$ be the proportions of the letters in $w$.
%, \ie, $q_a:=\frac{1}{m}\sumjm\indic{w_j=a}$.
Then
\begin{align}\label{kk}
\gss_1
\ge \frac{m^2}{n}\binom{n}{m}^2\norme{\vq-\vp}^2
={n}\binom{n-1}{m-1}^2\norme{\vq-\vp}^2.
\end{align}
\end{lemma}

\begin{proof}
  Let 
  \begin{align}\label{psi}
    \psi_i(x):=\sumjm c(i,j)\gf_{w_j}(x).
  \end{align}
Thus \eqref{vvi} is $\VVi=\psi_i(\xi_i)$, and \eqref{tauu} is,
since $\E\psi_i(\xi)=0$,
\begin{align}
  \gss_1
=\Var V_1
=\sumin\E\bigsqpar{\psi_i(\xi_i)^2}
=\E\sumin\psi_i(\xi)^2.
\end{align}
Hence, by the \CSineq,
\begin{align}\label{ka}
  n\gss_1
= % \ge 
n\E\sumin\psi_i(\xi)^2
\ge \E \Bigpar{\sumin\psi_i(\xi)}^2.
\end{align}
Furthermore, by \eqref{psi} and \eqref{sumc2}
\begin{align}\label{kb}
  \sumin \psi_i(x)
=\sumin\sumjm c(i,j)\gf_{w_j}(x)
=\sumjm \binom nm\gf_{w_j}(x)
=\binom{n}{m}\sumaa mq_a\gf_a(x).
\end{align}
Hence, \eqref{lgf3} yields
\begin{align}\label{kc}
\Bignormm{\sumin \psi_i(\xi)}
=m\binom{n}{m}\Bignormm{\sumaa q_a\gf_a(\xi)}
\ge m\binom{n}{m}\norme{\vq-\vp}.
\end{align}
Combining \eqref{ka} and \eqref{kc} yields \eqref{kk}.
\end{proof}

\section{Some Special Cases}
In this section we consider two interesting cases. 
In the first we assume that the pattern $w$ is alternating and
in the second case we consider random $w$. 

\subsection{Alternating $w$}\label{Salt}
\label{sec-low}
As an extreme example,
we consider alternating $w$, that is, $w=010101 \ldots$ for $\cA=\{0,1\}$.
%we prove in this subsection the estimate \eqref{alt} for an alternating
%string $w$. 
We prove that this case matches the general lower bound \eqref{l2} in \refL{L2}.
\begin{theorem}
  \label{TLlow}
Consider the unbiased binary case $\cA=\setoi$ and  $p_0=p_1=\frac12$,
and let $w$ be an alternating string $010101\dots$
Then, for any $m\le n/2$,
\begin{align}\label{alt2}
  \gss_1\le 10\frac{n}{m} \binom{n-1}{m-1}^2.
\end{align}
% CONJECTURE:
% for $i$ not too big, 
% $\tauu_i=\tauu_{n+1-i}\sim e^{-4im/n}c(1,1)^2$, and $\tauu_i$ ``small''
% when $i,n-i\gg n/m$, leading (by summing a geometric series) to 
\end{theorem}

\begin{proof}
  It is slightly more convenient to let $\cA=\set{\pm1}$; 
thus we consider $w=w_1\cdots w_m$ with 
\begin{align}
  \label{walt}
w_j=(-1)^j
\end{align}
in the unbiased case $p_1=p_{-1}=\frac12$. 
Then, by \eqref{gfa},
for $x\in\cA$,
%$  \gf_1(x)=2\indic{x=1}-1=x$ and
%$\gf_{-1}(x)=2\indic{x=-1}-1=-x$,
\begin{align}
  \gf_1(x)&=2 \cdot \indic{x=1}-1=x
\\
\gf_{-1}(x)&=2 \cdot \indic{x=-1}-1=-x,
\end{align}
and thus, for $a,x\in\cA$,
\begin{align}\label{ax}
  \gf_a(x)=ax.
\end{align}
%In particular, $\gf_a(x)\in\set{\pm1}$, and thus
%(or by \eqref{lgf1}),
%\begin{align}\label{cecilia}
%  \E\gf_a(\xi)^2=1.
%\end{align}

By \eqref{vvi}, \eqref{ax} and \eqref{walt},
\begin{align}\label{selma}
  V_{1,i}
=\sumjm c(i,j){w_j}\xi_i
%=\sumjm\binom{i-1}{j-1}\binom{n-i}{m-j}\gf_{w_j}\xi_i.
=\sumjm (-1)^jc(i,j)\xi_i
%=\sumjm(-1)^j\binom{i-1}{j-1}\binom{n-i}{m-j}\xi_i
=:\tau_i\xi_i,
\end{align}
where we thus define
\begin{align}\label{tau}
  \tau_i:=\sumjm (-1)^jc(i,j)
=\sumjm(-1)^j\binom{i-1}{j-1}\binom{n-i}{m-j}.
\end{align}
Note that \eqref{selma} gives % and \eqref{cecilia} give
$\E V_{1,i}^2=\tau_i^2$, 
so \eqref{tau} is consistent with our earlier definition  \eqref{tauu0}.
(The sign of $\tau_i$ is irrelevant for our purposes.)
By \eqref{tau} and \eqref{pi}--\eqref{hg}, we have, with
$\pi(i,j)$ and 
$X\sim\HG(n-1,m-1,i-1)$, as defined in Remark~\ref{RHG},
\begin{align}
\frac{  -\tau_i}{\binom{n-1}{m-1}} 
= \sumjm (-1)^{j-1}\pi(i,j)
=\sumjm(-1)^{j-1}\P(X=j-1)
=\E (-1)^X.
\end{align}
%where $\pi(i,j)$ and $X$ are defined in Remark~\ref{RHG}.
By \refL{LHG} below, this implies, for $2\le m\le n/2$ and $1\le i\le (n+1)/2$,
\begin{align}
  \frac{\abs{\tau_i}}{\binom{n-1}{m-1}} 
\le \exp \Bigpar{-\frac{(i-1)(n-i)(m-1)(n-m)}{(n-1)^2(n-2)}}
\le \exp \Bigpar{-\frac{(i-1)m}{8n}}.
\end{align}
This enables us to conclude, using the symmetry $|\tau_i|=|\tau_{n+1-i}|$ and
still assuming $2\le m\le n/2$, that
\begin{align}
\gss_1=
\sum_{i=1}^n\tau^2_i &
\le 2 \sum_{i=1}^{\ceil{n/2}}\tau^2_i 
\le2\binom{n-1}{m-1}^2 \sum_{i=1}^{\infty}e^{-(i-1)m/4n}
\notag\\&
=  \frac{2}{1-e^{-m/4n}} \binom{n-1}{m-1}^2
%\notag\\&
\le  \frac{10n}{m} \binom{n-1}{m-1}^2,
\end{align}
as claimed in \eqref{alt2}. 
The case $m=1$ is trivial by \eqref{lol}, with $B=1$ by \eqref{B}.
%\refR{Rlower}.
\end{proof}

\begin{lemma}\label{LHG}
Suppose that  $X$ is a hypergeometric random variable
 $X\sim\HG(n,k,\ell)$. Then
 \begin{align}\label{expo}
   \bigabs{\E(-1)^{X}} \le \exp\bigpar{-2\Var X}
=\exp\Bigpar{-2\frac{k(n-k)\ell(n-\ell)}{n^2(n-1)}}.
 \end{align}
\end{lemma}

Note that the expectation in \eqref{expo} is the difference of the
probabilities that $X$ is even or odd.

\begin{proof}
  By (a special case of) a theorem by \citet{VatutinM},
the probability generating function of $X$ has only negative real zeroes,
and thus there exist probabilities $r_i\in\oi$, $i=1,\dots,k$, such that
if $I_i\sim\Be(r_i)$ are independent indicator variables, then
$\sum_i I_i$ has the same distribution as $X$, \ie,
\begin{align}\label{vat}
  X \eqd \sum_i I_i.
\end{align}
Hence, with $s_i:=1-r_i$, 
\begin{align}
  \E(-1)^X = \E (-1)^{\sum_i I_i}
=\prod_i \E (-1)^{I_i}
=\prod_i (s_i-r_i)
\end{align}
and thus, using also $\Var X=\sum_i\Var I_i$ by \eqref{vat},
\begin{align}
\bigabs{  \E(-1)^X}
&=\prod_i |s_i-r_i|
=\prod_i(1-2\min\set{r_i,s_i})
\le \prod_i(1-2{r_is_i})
\notag\\&
\le \exp\Bigpar{-2\sum_i r_is_i}
=\exp\Bigpar{-2\sum_i\Var I_i}
\notag\\&
=\exp\bigpar{-2\Var X}.
\end{align}
This yields \eqref{expo} by
the standard formula
\begin{align}
  \label{varHG}
\Var X = \frac{k(n-k)\ell(n-\ell)}{n^2(n-1)}.
\end{align}
This completes the proof.
\end{proof}

\subsection{A random $w$}
\label{sec-random}

\refT{Tka} applies when $w$ is far from a typical string $\Xi_m$ from our
random source.
In this subsection we consider the opposite case, \ie, when $w$ is
like $\Xi_m$.
More precisely, we consider the case when $w=W$ is a random string, of a given
length $m$, drawn from the same source; thus $W\eqd\Xi_m$, but $W$
is independent of $\Xi_n$.
(We use capital $W$ to emphasize that the string is random.)
We think of this as a two-stage random experiment. First 
we sample $W$; then we sample $\Xi$. Conditioned on $W=w$, we thus have the
same situation as before. 

We write, for example, $\gss_1(w)$ to indicate the
dependence on $w$; thus $\gss_1(W)$ is a random variable.
The next theorem shows that $\gss_1(W)$ is concentrated about a value that is
roughly the  geometric mean of the upper and lower bounds in \eqref{lol} and
\eqref{l2}. 
\begin{theorem}
  \label{TLrandom}
Let $W\eqd\Xi_m$. Then, for $n\ge1$ and $1\le m\le n/2$, 
\begin{align}\label{lrandom}
  \E[ \gss_1(W)] = \Theta\biggpar{\frac{n}{\sqrt m}\binom{n-1}{m-1}^2}.
\end{align}
Furthermore, if also $m,n\to\infty$, then
\begin{align}\label{lprandom}
  \frac{\gss_1(W)}{  \E[ \gss_1(W)]} \to1
\end{align}
in probability. 
\end{theorem}

\begin{corollary}
\label{cor-random}
For random $w=W\eqd\Xi_m$, \eqref{t1b} holds for $m=o\bigpar{n^{2/5}}$
with high probability,
and hence
for a typical pattern $w$ the number of $w$ occurrences $Z$ is
asymptotically normal 
as long as $m=o\bigpar{n^{2/5}}$. More precisely, in this case
\begin{equation}
  Z/\E Z \sim \ASN\Bigpar{1,\frac{\E[ \gss_1(W)]}{\E^2[ Z]}} 
\end{equation}
with $\E [Z] = \binom{n}{m} 2^{-mh + \Op(m\qq)}$ 
where $h=-\sum_{a\in \cA} p_a \log p_a$ is the source entropy.
\end{corollary}

%\begin{comment}
%In other words, for a random word $w=W$, with high probability,
%$\gss_1(w)$ is of the order \eqref{lrandom}, and thus
%\eqref{t1b} holds for $m=o\bigpar{n^{2/5}}$.
 %Hence, somewhat informally, \refT{T1} shows that
%for a typical random $w$,
%with high probability, 
%$Z$ is asymptotically normal 
%as long as $m=o\bigpar{n^{2/5}}$.
%(More formally, the conditional distribution of $(Z-\E Z)/(\Var Z)\qq$ given
%$W$ converges to $N(0,1)$ in probability.) 
%Moreover, for such $w$, 
%\refConj{Conj1} conjectures that asymptotic normality
%holds for $m=o\bigpar{n^{2/3}}$, and log-normality beyond that.
%\end{comment}

\begin{proof}[Proof of Theorem~\ref{TLrandom}]
  Define the covariance matrix
\begin{align}
  \rho(a,b):=\Cov\bigpar{\gf_a(\xi),\gf_b(\xi)}
%=p_a\qw p_b\qw\Cov\bigpar{\indic{\xi=a},\indic{\xi=b}}
%=p_a\qw\indic{a=b}-1  
,\qquad a,b\in\cA.
\end{align}
We have already computed $\rho(a,a)=p_a\qw-1$ in \eqref{lgf1p}.
Similarly, in general,
 recalling \eqref{gfa},
\begin{align}\label{rhoab}
  \rho(a,b)
%:=\Cov\bigpar{\gf_a(\xi),\gf_b(\xi)}
=p_a\qw p_b\qw\Cov\bigpar{\indic{\xi=a},\indic{\xi=b}}
=p_a\qw\indic{a=b}-1  
%\qquad a,b\in\cA
.\end{align}

By \eqref{vvi}, for a given string $w$,
\begin{align}
  \tauu_i(w)
&=\Var \VV{i} 
=\sumjm\sumkm c(i,j)c(i,k)\Cov\bigpar{\gf_{w_j}(\xi_i),\gf_{w_k}(\xi_i)}
\notag\\&
=\sumjm\sumkm c(i,j)c(i,k)\rho\bigpar{w_j,w_k},
\end{align}
where $\rho\bigpar{w_j,w_k}=\Cov\bigpar{\gf_{w_j}(\xi_i),\gf_{w_k}(\xi_i)}$.
Thus, by \eqref{tauu},
\begin{align}\label{wille}
\gss_1(w)
&=\sumin  \tauu_i(w)
%\notag\\&
=\sumin\sumjm\sumkm c(i,j)c(i,k)\rho\bigpar{w_j,w_k}.
\end{align}

Now let $w=W$ be random, with $W\eqd\Xi_m$. Then, the letters
$W_j$ are \iid{} with $W_j\eqd \xi$.
In particular, it follows from \eqref{rhoab} that for any fixed $a$,
\begin{align}\label{willu}
  \E \rho(W_j,a)=\E \rho(a,W_j)=0,
\end{align}
and thus $\E\rho(W_j,W_k)=0$ when $j\neq k$, while
\begin{align}\label{willw}
  \E\rho(W_j,W_j)=\sumaa p_a\bigpar{p_a\qw-1}=|\cA|-1 =: A_1.
\end{align}
Consequently, taking the expectation in \eqref{wille} and recalling \eqref{pi},
\begin{align}\label{willy}
\E[ \gss_1(W)]
&
%\notag\\&
=\sumin\sumjm\sumkm c(i,j)c(i,k)\E\rho\bigpar{W_j,W_k}
=A_1\sumin\sumjm c(i,j)^2
\notag\\&
=A_1\binom{n-1}{m-1}^2\sumin\sumjm \pi(i,j)^2.
\end{align}
where $\pi(i,j)$ is defined in Remark~\ref{RHG}.
We thus want to estimate the final double sum.

First, fix $i$ and recall from \eqref{hg} that $(\pi(i,j))_j$ is the
probability distribution of $X+1$ with $X\sim\HG(n-1,m-1,i-1)$.
Let $\mu:=\E X+1$ and $\gamm:=\Var X$.
By Chebyshev's inequality,
\begin{align}
    \sum_{|j-\mu|> 2\gam}\pi(i,j)=\P\bigpar{|X+1-\mu|>2\gam}\le \frac{1}4,
\end{align}
and thus by the \CSineq,
\begin{align}
  \frac{9}{16}
\le \Bigpar{\sum_{|j-\mu|\le 2\gam}\pi(i,j)}^2
\le (4\gam+1)\sum_{|j-\mu|\le 2\gam}\pi(i,j)^2.
\end{align}
Furthermore, see \eqref{varHG},
$\gamm=\Var X\le im/n$. Hence,
\begin{align}
  \sumjm\pi(i,j)^2 
\ge \frac{C}{\gam+1}
\ge C \min\bigpar{\gam\qw,1}
\ge C \min\Bigpar{\Bigparfrac{n}{mi}\qq,1}.
\end{align}
Summing over $n/2\le i \le n$, say, yields
\begin{align}\label{goa}
  \sumin  \sumjm\pi(i,j)^2 \ge C\frac{n}{m\qq}.
\end{align}

In the opposite direction, we again fix $i$ and note that
\begin{align}\label{gobi}
    \sumjm\pi(i,j)^2 
\le \max_j\pi(i,j) \sumjm\pi(i,j)
= \max_j\pi(i,j).
\end{align}
It follows from \eqref{cij} that 
\begin{align}\label{gog}
\frac{\pi(i,j+1)}{\pi(i,j)}
=\frac{c(i,j+1)}{c(i,j)}
=\frac{(i-j)(m-j)}{j(n-i-m+j+1)}  ,
\end{align}
and it follows easily that  the maximum in \eqref{gobi}
is attained at 
\begin{equation}
\label{eq-ws1}
j=j_0:=\lrceil{\frac{im}{n+1}}=\frac{im}{n}+O(1).
\end{equation}
It is then easy to see, by Stirling's formula and some calculations, that
for $i\le \ceil{n/2}$,
\begin{align}\label{gor}
 \max_j\pi(i,j) \le C\Bigparfrac{n}{mi}\qq.
\end{align}
Hence, by \eqref{gobi} and \eqref{gor},
\begin{align}\label{goj}
\sumin \sumjm\pi(i,j)^2 
\le 2\sum_{i=1}^{\ceil{n/2}} \max_j\pi(i,j)
\le C \sum_{i=1}^n \frac{n\qq}{m\qq i\qq}  
\le C  \frac{n}{m\qq}.  
\end{align}
The result \eqref{lrandom} for the expectation follows by
\eqref{willy}, \eqref{goa} and \eqref{goj}.

Next, we estimate the variance of $\gss_1(W)$.
Let
\begin{align}\label{Yx}
Y:=  \gss_1(W)\bigm/\binom{n-1}{m-1}^2
=\sumjm\sumkn\sumin\pi(i,j)\pi(i,k)\rho(W_j,W_k)
\end{align}
and note that, by \eqref{lrandom},
\begin{align}\label{EY}
  \E Y = \Theta\Bigparfrac{n}{m\qq}.
\end{align}
Since the random letters $W_j$ are independent,
it follows from \eqref{willu} that
the random variables $\rho(W_j,W_k)$, $j\le k$, have
covariances 0; furthermore, these variables are bounded.
Hence, \eqref{Yx} implies
\begin{align}
  \Var Y& 
\le C\sumjm\sumkm \Bigpar{\sumin\pi(i,j)\pi(i,k)}^2\label{vary}
.\end{align}
To estimate \eqref{vary}, we split  the inner sum
into the ranges $i\le \ceil{n/2}$ and $i>\ceil{n/2}$, using $(x+y)^2\le
2(x^2+y^2)$; by symmetry it suffices to consider the case $i\le\ceil{n/2}$.
%(Which we write as $i\le n/2$ in the sequel.)
It follows from \eqref{gog} after some calculations that then
\begin{align}
  \pi(i,j)
& \le 
C e^{-C(j-j_0)^2/(j+j_0)}\pi(i,j_0)
\le C j_0\qqw e^{-C(j-j_0)^2/(j+j_0)}
\notag\\&
\le C j\qqw e^{-C(j-j_0)^2/(j+j_0)}
%\le C \bigparfrac{n}{mi}\qq \bigpar{e^{-C(im-jn)^2/(n^2j+imn)}}
\end{align}
where $j_0$ is defined in (\ref{eq-ws1}).
It follows, omitting the details, that for $1\le j\le k\le m$,
\begin{align}
  \sum_{i=1}^{\ceil{n/2}}\pi(i,j)\pi(i,k)
\le C  \frac{n}{m k\qq} e^{-C(j-k)^2/m}
\end{align}
and thus \eqref{vary} yields, using \eqref{EY},
\begin{align} 
 \Var Y&
\le C \sumkm \sumjm \frac{n^2}{m^2 k} e^{-C(j-k)^2/m}
\le C \sumkm  \frac{n^2}{m^{3/2} k}
%\notag\\&
\le C  \frac{n^2}{m^{3/2}}\log m
\notag\\&
\le C \frac{\log m}{m\qq}(\E Y)^2.
\end{align}
Consequently, as \mtoo,
\begin{align}
  \Var\Bigparfrac{\gss_1(W)}{\E[\gss_1(W)]}
=\Var\Bigparfrac{Y}{\E Y} 
\le C \frac{\log m}{m\qq}\to0,
\end{align}
and \eqref{lprandom} follows.
\end{proof}

\section{Concluding Remarks}

Finally, we collect here some further comments, examples and conjectures,
in the hope of stimulating further research.

\begin{example}\label{Eaaa2}
Consider again the case when $w=a^m=a\dotsm a$ is a constant string, treated by a
direct method in \refS{Saaa} and \refT{TLN}.
Let us see what \refT{T1} yields.
In this case, by \eqref{cec},
with $c_a:=p_a\qw-1>0$,
%\begin{align}
%  \VVi:=
%%\sumjm\binom{i-1}{j-1}\binom{n-i}{m-j}\gf_a(\xi_i)
%\sumjm c(i,j)\gf_a(\xi_i)
%=\binom{n-1}{m-1}\gf_a(\xi_i)
%\end{align}
%and thus, see \eqref{lgf1} and let $c_a:=p_a\qw-1>0$,
%\begin{align}
%\tauu_i
%=\Var V_{1,i}
%=\binom{n-1}{m-1}^2 \Var\bigpar{\gf_a(\xi_i)}
%=c_a\binom{n-1}{m-1}^2 .
%\end{align}
%Hence,
\begin{align}\label{and}
  \gssw=\sumin\tauu_i = nc_a\binom{n-1}{m-1}^2 ,
\end{align}
and thus \eqref{t1b} reduces to $m^2=o(n)$.
(This also follows by \refL{Lk}.)

Consequently, \refT{T1} applies and shows asymptotic normality when
$m=o\bigpar{n\qq}$, which we already knew, see
Theorems \ref{TLN}\ref{TLN3} and \ref{Tka}.
This example shows that Theorems \ref{T1} and \ref{Tka} are sharp,
in the sense that  the range of $m$ for which they yield asymptotic normality
cannot be extended; see \refE{Eaaa}.
\end{example}

\begin{remark}\label{Raaf}
The argument in the proof of \refT{Tka} applies
also in other cases  where
$\gss_1$ is of the same order as the upper bound in \eqref{lol}.
Then \refT{T1} applies and shows asymptotic normality for $m=o\bigpar{n\qq}$.
A simple example is when
$w=0\dotsm01\dotsm1$, or more generally, when, say, the first and second
half of $w$ have different distributions of the letters, even if the average
proportions in the entire string $\vq=\vp$.
(This can be seen by a modification of the argument in the proof of \refL{Lk}.)
\end{remark}

Based on these examples
we conjecture  the following.

\begin{conjecture}\label{Conj1}
If\/ $\gss_1=o\bigpar{\binom{n}{m}^2}$,
or equivalently
$\gss_1=o\bigpar{\frac{n^2}{m^2}\binom{n-1}{m-1}^2}$,
then
\begin{align}\label{conj1}
  \Z/\E\Z\sim
\ASN\Bigpar{1,\frac{\gss_1}{\binom{n}{m}^2}}.
\end{align}
Moreover, at least as long as $m=o(n)$,
\begin{align}\label{conj1b}
\ln  \Z\sim
\ASN\Bigpar{a_n,\frac{\gss_1}{\binom{n}{m}^2}}
\end{align}
for some sequence $a_n$.
\end{conjecture}

In particular, by \eqref{alt2}, if \refConj{Conj1} holds,
then for an alternating string
$w=0101\dotsm$, $Z$ is asymptotically normal for any $m=o(n)$.
Moreover, for random $w$ as discussed in Section~\ref{sec-random}, 
by \refT{TLrandom},
\refConj{Conj1} suggests that asymptotic normality
holds for $m=o\bigpar{n^{2/3}}$, and log-normality beyond that.

Note that
this conjecture implies that if $\gss_1$ is of a smaller order
than the upper bound in \eqref{lol} (for $n\qqq\le m\le n$, say),
then asymptotic normality holds for a
larger range of $m$ than $o\bigpar{n\qq}$,
while our proof above, on the contrary,
verifies this only in a range smaller than $m=o\bigpar{n\qq}$.

\newcommand\AMS{Amer. Math. Soc.}
\newcommand\Springer{Springer-Verlag}
\newcommand\Wiley{Wiley}

\newcommand\vol{\textbf}
\newcommand\jour{\emph}
\newcommand\book{\emph}
\newcommand\inbook{\emph}
\def\no#1#2,{\unskip#2, no. #1,} %(typeset after year) 
\newcommand\toappear{\unskip, to appear}

\newcommand\arxiv[1]{\texttt{arXiv}:#1}
\newcommand\arXiv{\arxiv}

\def\nobibitem#1\par{}

%\newpage


\begin{thebibliography}{99}

%\bibitem[??]{??} 

\bibitem{ags03}
M. Atallah, R. Gwadera and W. Szpankowski,
Reliable detection of episodes in event sequences.
{\it Third IEEE International Conference on Data Mining} (ICDM-03),  67--74,
Melbourne,  Florida,  November 2003.

\bibitem{BeKo93}
{E.~A. Bender  and F. Kochman},
\newblock The distribution of subword counts is usually normal.
\newblock {\em European Journal of Combinatorics} 
\vol{14} (1993), 265--275.

\bibitem{BoVa02}
J. Bourdon and B. Vall\'ee,
Generalized pattern matching statistics.
In {\em Mathematics and Computer Science} (Colloquium
Proceedings, Versailles, 2002), B. Chauvin et al. eds.,
Birkh{\"a}user Verlag, 2002, pp.~229--245.

\bibitem{olgica19}
M. Cheraghchi, R. Gabrys, O. Milenkovic and J. Ribeiro,
Coded trace reconstruction.
Preprint, 2019.
\arXiv{1903.09992}

\bibitem{dsv12}
M. Drmota, K. Viswanathan and W. Szpankowski,
Mutual Information for a Deletion Channel,
{\it ISIT 2012}, Boston, 2012.


\bibitem{dg06}
S. Diggavi and M. Grossglauser,
Information transmission over finite buffer channels,
{\it IEEE Trans. Info. Th.}, \vol{52} (2006), 1226--1237.

\bibitem{dobrushin67}
R. L. Dobrushin,
Shannon's theorem for channels with synchronization errors.
%{\it Prob. Info. Trans.}, 18-36, 1967.  %original pagination??
\emph{Problems Inform. Transmission} \vol3 (1967), no. 4, 11--26 (1969). 
%MR0289198


\bibitem{fsv06}
P.~Flajolet, W.~Szpankowski and B.~Vall\'ee,
\newblock Hidden word statistics.
\newblock {\em Journal of the ACM} \vol{53} (2006), 1--37.

\bibitem{Gut}
A. Gut, %Allan
\emph{Probability: A Graduate Course},
2nd ed. Springer, New York, 2013. 

%\bibitem[Diananda(1955)]{Diananda}
%P. H. Diananda, %MR
%The central limit theorem for $m$-dependent variables.
%\emph{Proc. Cambridge Philos. Soc.} \textbf{51} (1955), 92--95.

\bibitem[Hoeffding(1948)]{Hoeffding}
W. Hoeffding, % Wassily
A class of statistics with asymptotically normal  distribution. 
\emph{Ann. Math. Statistics} \textbf{19} (1948), 293--325.


%\bibitem[Hoeffding and Robbins(1948)]{HoeffdingR}
%Wassily Hoeffding and
%Herbert Robbins,
%The central limit theorem for dependent random variables.
%\emph{Duke Math. J.} \textbf{15} (1948), 773--780.

\bibitem{holden}
N. Holden and R. Lyones, 
Lower bounds for trace reconstruction.
Preprint, 2018.
\arXiv{1808.02336}

\bibitem{js-book}
P. Jacquet and W. Szpankowski,
{\it Analytic Pattern Matching: From DNA to Twitter}.
Cambridge University Press, 2015.

\bibitem{SJ287}
S. Janson, B. Nakamura and D. Zeilberger.
% Svante Janson, Brian Nakamura and Doron Zeilberger.
On the asymptotic statistics of the number of occurrences of multiple
permutation patterns. 
\emph{J. Comb.} \textbf6 (2015), no. 1-2, 117--143. 

\bibitem{kms10}
A. Kalai, M. Mitzenmacher and M. Sudan,
Tight asymptotic bounds for the deletion channel
with small deletion probabilities.
{\it ISIT}, Austin, 2010.

\bibitem{KanMont}
Y. Kanoria and A. Montanari,
On the deletion channel with small deletion probability.
{\it ISIT}, Austin, 2010; see \arXiv{1104.5546} for an extension.

\bibitem{mcgregor14}
A. McGregor, E. Price and S. Vorotnikova,
Trace reconstruction revisited.
{\it European Symposium on Algorithms} (2014), 689--700.

\bibitem{Mitz}
M. Mitzenmacher,  %Michael
A survey of results for deletion channels and related
synchronization channels. 
{\it Probab. Surveys} \vol6 (2009), 1--33.

\bibitem{peres17}
Y. Peres and A. Zhai,
Average-case reconstruction for the deletion channel:
subpolynomially many traces suffice.
{\it FOCS}, 2017.

\bibitem[Vatutin and Mikha{\u\i}lov(1982)]{VatutinM}
V. A. Vatutin and V. G. Mikha{\u\i}lov, %Mikhailov SIAM
%Limit theorems for the number of empty cells in an equiprobable scheme for
%the distribution of particles by groups. (Russian) 
Limit theorems for the number of empty cells in an equiprobable scheme for
group allocation of particles.  (Russian) 
\emph{Teor. Veroyatnost. i Primenen.} \vol{27} (1982),
no. 4, 684--692. 
English transl.:
\emph{Theory Probab. Appl.} \vol{27} (1983), no. 4, 734--743.
%https://doi.org/10.1137/1127084
%MR0681461


\bibitem{vtr11}
R. Venkataramanan, S. Tatikonda and K. Ramchandran,
Achievable rates for channels with deletions and insertions.
{\it ISIT}, St. Petersburg, Russia, 2011.



\end{thebibliography}
\end{document}